\documentclass[11pt,DIV=12]{scrartcl}
\usepackage{amsmath,amsthm}
\usepackage{newlfont,mathrsfs,enumerate,amssymb,array,amscd,xcolor}

\newcommand{\mini}{{\mathrm{mini}}}

\newcommand{\bbC}{{\mathbb C}}

\newcommand{\Aut}{{\mathrm{Aut}}}

\newcommand{\Ind}{{\mathrm{Ind}}}

\newcommand{\ind}{{\mathrm{ind}}}

\newcommand{\rO}{{\mathrm{O}}}

\newcommand{\Spin}{{\mathrm{Spin}}}
\newcommand{\SL}{{\mathrm{SL}}}
\newcommand{\Sp}{{\mathrm{Sp}}}
\newcommand{\GL}{{\mathrm{GL}}}
\newcommand{\PGL}{{\mathrm{PGL}}}

\newcommand{\SO}{{\mathrm{SO}}}
\newcommand{\St}{{\mathrm{St}}}

\newcommand{\syp}{\widetilde{\mathrm{Sp}}}

\newcommand{\Hom}{{\mathrm{Hom}}}

\newtheorem{lemma}{Lemma}[section] 

\newtheorem{prop}[lemma]{Proposition}
\newtheorem{thm}[lemma]{Theorem}
\newtheorem{cor}[lemma]{Corollary}

\makeatletter
\def\blfootnote{\gdef\@thefnmark{}\@footnotetext}
\makeatother

\begin{document}


\title{Multiplicity free Weil representations arising from exceptional groups}
\author{Marcela Hanzer\thanks{hanmar@math.hr}\quad and Gordan Savin\thanks{gordan.savin@utah.edu}}

\date{
    \footnotesize{University of Zagreb $\quad $ University of Utah}
    }


\maketitle

\blfootnote{\textup{2020} \textit{Mathematics Subject Classification}:
22E46, 22E47
\newline This work is  supported in part by the Croatian Science Foundation under the project IP-2022-10-4615 and by a gift No. 946504
from the Simons Foundation.}
\begin{abstract}\noindent\textbf{Abstract} 
  Using exceptional theta correspondences, we prove that certain Weil representations of $p$-adic groups are multiplicity free 
  and determine irreducible quotients. 
\end{abstract}

\section{Introduction}  \label{S:introduction} 
Let $F$ be a $p$-adic field. In this paper we consider a split simply connected group $\mathscr G$ over $F$, given by the Chevalley-Steinberg construction, 
starting from a root system $\Phi$ in the sequence 
\[ 
D_4 < E_6 < E_7 < E_8. 
\] 
Let $\Delta\subset \Phi$ be a set of simple roots. 
 Let $\varphi : \SL_2 \rightarrow  \mathscr G$ be the root $\SL_2$ corresponding to the highest root $\beta$.  
 Let $U$ and $\bar U$ be  the groups of upper and lover triangular unipotent matrices in $\SL_2$ which, under $\varphi$, 
 are identified with roots spaces for $\beta$ and $-\beta$.  The co-root $\beta^{\vee}$ is given 
 by 
 \[ 
 \beta^{\vee}(a) = \varphi \left( \begin{array}{cc} 
 a & 0 \\
 0 & a^{-1} 
 \end{array} \right) . 
 \] 
  There is a unique simple root $\alpha$ such that $\langle \alpha, \beta\rangle =1$.
   Let $\mathscr Q$ be the maximal parabolic 
  subgroup corresponding to $\alpha$. Its unipotent radical $\mathscr U$ is a Heisenberg group and $U$ is its center. 
  Similarly,  let $\bar{\mathscr Q}$ be the opposite maximal parabolic, and $\bar{\mathscr U}$ its radical. The intersection of ${\mathscr Q}$  and 
  $\bar{\mathscr Q}$ is a standard Levi. Its center is the image of the co-root $\beta^{\vee}$. Let $G$ be the derived group of the Levi. It is a simply 
  connected semi-simple group corresponding to $\Delta \setminus \{\alpha\}$.  The centralizer of $G$ in $\mathscr G$ is $\SL_2$ and vice versa. 
  Thus $(\SL_2, G)$ is a dual pair in $\mathscr G$.  
  
   Now observe that $V={\mathscr U}/ U$ is an irreducible symplectic  representation of $G$ with the (minuscule) highest weight $\beta -\alpha$. 
   The isomorphism class of the pair $(G,V)$ is given in the following table.  
 \[ 
 \begin{array}{c||c|c|c|c} 
 \mathscr G & D_4 & E_6 & E_7 & E_8  \\ \hline 
G & \SL_2^3 & {\SL_6} & \Spin_{12} & E_7  \\ \hline 
V & V_2^{\otimes 3} & \wedge^3 V_6 & V_{32} & {V_{56}}   \\ 
 \end{array}. 
 \] 
 Now fix an additive character $\psi$ of $U$. Let $\Omega_{\psi}$ be the unique irreducible representation 
 of $\mathscr U$ with the central character $\psi$. By \cite[Proposition 2]{KS_minimal} the group $G$ acts on $\Omega_{\psi}$ without need to pass to a central extension. 
 This is the Weil representation in the title of  this paper. We note that the pairs $(G,V)$ appear in Table 1 of \cite{Knop_symplectic} 
 and decomposing $\Omega_{\psi}$ is a now a part of a program \cite{BZSV}, see \cite{MWZ} and Conjecture 2.3 there, in particular.  
 In fact, Mao, Wan and Zhang study this problem using the relative trace formula, and indicate a possibility of the approach via minimal representations 
\cite[Remark 2.4]{MWZ}. This is what we do in this paper.  
 \vskip 10pt

 More precisely,  in Theorem \ref{T:mult_free}, we prove that for every irreducible representation $\pi$ of $G$, 
 \[ 
\dim  \Hom_G(\Omega_{\psi}, \pi) \leq 1 
\] 
and classify $\pi$ that appear as quotients of $\Omega_{\psi}$.  To that end we completely determine theta correspondence for the dual pair $(\SL_2, G)$. 
The analysis of $\Omega_{\psi}$ is then a simple consequence since $\Pi_{U,\psi}\cong \Omega_{\psi},$ where $\Pi$ is the minimal representation of $\mathscr G.$ 
In words, $\pi$ is a quotient of $\Omega_{\psi}$ if and only 
if it is a theta lift of a $(U,\psi)$-generic irreducible representation $\sigma$ of $\SL_2$.  
Thus most of the paper is devoted to computing the theta correspondence for the dual pair $\SL_2\times G$. 

\vskip 10pt 

From now on, for an irreducible representation $\sigma$ of $\SL_2(F)$, we denote by $\Theta(\sigma)$ the corresponding big theta lift to $G$: 
\[ 
\Theta(\sigma):= (\Pi\otimes \sigma^{\vee})_{\SL_2}. 
\] 
 Observe that any linear functional on $\Theta(\sigma)$ naturally corresponds to an $\SL_2$-invariant  bilinear functional on $\Pi \times \sigma^{\vee}$. 
 Such a functional, in turn, corresponds to an element in $\Hom_{\SL_2}(\Pi, \sigma)$.   
 In Section \ref{S:principal}, we show that  theta lifts of principal series representations of $\SL_2$ are degenerate principal series representations of $G$ 
 induced from a Siegel parabolic $P=MN$ in $G$.  In Section  \ref{S:exceptional}, using a technique of Fourier-Jacobi functors,  we prove that the correspondence 
 is one to one. 
 For all tempered representations $\sigma$ we prove that $\Theta(\sigma)$ is irreducible and 
 express it as a Langlands quotient using a tower of theta lifts, staring from the dual pair in (adjoint) $D_4$. There we have a 
 dual pair 
 \[ 
 \SL_2 \times (\SL_2^3/\mu_2^2) 
 \] 
 and we call the correspondence ``mini-theta'' since it is a building block for the other three correspondences.  
 Let $\pi$ be an irreducible representation of $\GL_2(F)$ such that the restriction to $\SL_2(F)$  contains  $\sigma$. 
 Observe that $\pi\otimes \pi \otimes \pi$ is naturally a representation of  the group $(\SL_2^3/\mu_2^2) (F)$. 
  Assume then $\sigma\neq 1$. Let 
 \[ 
 \mathrm{mini}(\sigma)= (\Pi\otimes \sigma^{\vee})_{\SL_2}.
 \] 
 In Section \ref{S:mini_theta}  we show that 
 $ \mathrm{mini}(\sigma)$ is an irreducible representation of $(\SL_2^3/\mu_2^2) (F)$ contained in $\pi\otimes \pi \otimes \pi$. 
 Moreover, mini-theta gives a one-to-one correspondence between irreducible summands of $\pi$ and irreducible summands of $\pi\otimes \pi \otimes \pi$.

 \vskip 10pt 
 With the mini-theta in place, we proceed to describe the parabolic $Q=LU$ of type $A_1^3$ in $G$. Consider the marked Dynkin diagram for $E_7$ 
 
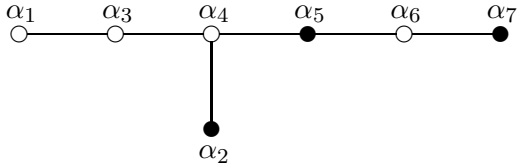
\begin{picture}(200,120)(-120,-15)

\put(79,73){\line(0,-1){30}}
\put(79,40){\circle*{6}}
\put(154,76){\line(1,0){30}}
\put(187,76){\circle*{6}}

\put(74,29){$\alpha_2$}

\put(02,82){$\alpha_1$}

\put(38,82){$\alpha_3$}

\put(74,82){$\alpha_4$}

\put(110,82){$\alpha_5$}

\put(146,82){$\alpha_6$}

\put(182,82){$\alpha_7$}

\put(07,76){\circle{6}}
\put(10,76){\line(1,0){30}}

\put(43,76){\circle{6}}
\put(46,76){\line(1,0){30}}
\put(79,76){\circle{6}}

\put(82,76){\line(1,0){30}}
\put(115,76){\circle*{6}}

\put(118,76){\line(1,0){30}}
\put(151,76){\circle{6}}

\end{picture}

We consider  $G$ to be in the sequence $A_5 < D_6 <E_7$ where $A_5$ is picked to contain the three black dots.  
Let $Q=LU$ be the standard parabolic in $G$ corresponding to the three simple roots, that is, $[L,L]=\SL_2^3$. 
Let $S$ be the connected component of the center of $L$. So it is a split torus of dimension $2,3$ and $4$, respectively. Then, 
remarkably, $L/S\cong \SL_2^3/\mu_2^2$ (as algebraic groups) so $ \mathrm{mini}(\sigma)$ is a representation of  $L$. 
Then $\Theta(\sigma)$, the lift to $G$, is the Langlands quotient of $\Ind_Q^G( \mathrm{mini}(\sigma)\otimes \chi)$ for 
a positive character $\chi$ of $L$. More details are in Section \ref{S:parameters}, here we repackage  the result to confirm 
a local variant of  Conjecture 2.3 in \cite{MWZ}. 

\vskip 5pt 
 
 Recall that the dual group of $\SL_2$ is $\PGL_2(\mathbb C)$. 
 Let $\varphi : WD_F \rightarrow \PGL_2(\mathbb C)$ be the parameter of $\sigma$.  Since $G$ is simply connected and simply laced, 
 the dual group is $G_{ad}(\mathbb C)$, the identity component of $\Aut(G)$.  We shall omit $\mathbb C$ in our notation henceforth. 
 Recall, from \cite{KS_jordan}  that $G_{ad}$ arises via the Koecher-Tits construction starting with a Jordan algebra $J$. We have 
  have a dual pair 
 \[ 
 \PGL_2 \times \Aut(J) \subset \Aut(G). 
 \] 
 The parameter $\varphi_G$ of $\Theta(\sigma)$ is obtained in the following two steps. First, using $\varphi$, we map  $WD_F$ into 
  the first factor above. Second, we twist the Frobenius using the principal $\SL_2$ in $\Aut(J)$, 
 $\Psi :\SL_2(\mathbb C) \rightarrow \Aut(J)$: 
\[ 
\varphi_G(\mathrm{Fr}) = \varphi(\mathrm{Fr}) \times \Psi\left(\begin{smallmatrix}q^{1/2} & 0 \\ 0 & q^{-1/2} \end{smallmatrix}\right). 
\]

\section{Theta correspondence for principal series} \label{S:principal}  

 Let $B=TU$  be a Borel subgroup of upper triangular matrices in $\SL_2$. Let $t=(a,a^{-1})$  be a diagonal matrix in $T$.  
 Then $\delta_{B}(t)^{1/2}=|a|$ is the modular character.  If $\chi$ is a character of $F^{\times}$, let 
  $\Ind_B^{\SL_2}(\chi)$ denote the normalized induced representation. The goal of this section is to (almost) compute big-theta for 
  constituents of principal series representations. 
   
 \medskip 
 
 \subsection{A degenerate principal series for $G$}  Let $(G,V)$ be a pair as above, but not $(\SL_2^3, V_2^{\otimes 3}$). Let $\varpi$ be the highest weight of $V$, and let 
 $P=MN$ be the maximal parabolic subgroup of $G$ stabilizing the highest weight line. 
 Then $M$ acts on the line by the character which we denote by $\varpi$, 
 a slight abuse of notation.  The weight $\varpi$ is a minuscule highest weight, in particular, $N$ is abelian. Its dimension is $3+6r$ where $r=1,2,4$ in the three 
 cases.  The cases are given by the following table, where in the last row we explain how $M^{\mathrm{der}}$  acts on $N$ using the ``standard'' representations. 
  \[ 
 \begin{array}{c||c|c|c} 
 \mathscr G &  E_6 & E_7 & E_8  \\ \hline 
G &  {\SL_6} & \Spin_{12} & E_7  \\ \hline 
M^{\mathrm {der}} &  \SL_3\times \SL_3  & \SL_6 &  E_6  \\ \hline 
N &  V_3\otimes V^{\ast}_3  & \wedge^2 V_6 &  V_{27}  \\ 
 \end{array}. 
 \]

  Let $\chi$ be a character of $F^{\times}$. 
 The composite $\chi\circ \varpi$ gives a character of $M$ and a normalized degenerate principal series $\Ind_P^G(\chi)$. Reducibility points and composition series 
 have been determined by Weissman \cite{Weissman_FJ_small}. Observe that $\Ind_P^G(\chi^{-1})$ is the contragredient of $\Ind_P^G(\chi)$: 
 
 \begin{prop} \label{P:degenerate_ps} 
  Let $r=1,2,4$, so that $\dim N=3+6r$.  Then $\Ind_P^G(\chi)$ is irreducible unless 
 \begin{itemize} 
 \item $\chi$ is quadratic and non-trivial. Then $\Ind_P^G(\chi)$ is a direct sum of two irreducible representations. 
 \item $\chi=|\cdot|^s$ where $\pm s=1, 1+r, 1+2r$, when $\Ind_P^G(\chi)$ has length 2, with unique irreducible quotient and irreducible sub. 
 \end{itemize} 
 The trivial representation $1_G$ is the unique quotient for $s=1+2r$. The minimal representation $\Pi_G$ is the unique quotient for $s=1+r$.  
  \end{prop}

 \subsection{Minimal representation} 
 \label{minimal_rep} 
 Let $\Pi$ be the minimal representation of $\mathscr G$ (\cite{KS_minimal}, \cite{GS_minimal}). Let $\Omega\subset V$ be the $G$-orbit  of non-zero highest weight vectors. 
 By  \cite[Theorem 6.1]{MS} we have an exact sequence of $\mathscr Q$-modules 
 \[ 
 0\rightarrow C_c(\Omega) \rightarrow \Pi_{U} \rightarrow \Pi_{\mathscr U} \rightarrow 0 
 \] 
 where $t\in T$ acts on $f\in C_c(\Omega)$ by 
 \[ 
 |a|^{2+2r} f(ax) 
 \] 
 and $g\in G$ by 
  \[ 
 f(g^{-1}x). 
 \]

 Furthermore, $\Pi_{\mathscr U}$, as a $G$-module, is a direct sum 
 \[ 
 \Pi_{\mathscr U} \cong \Pi[G] \otimes |a|^{2+r} \oplus \mathbb C \otimes |a|^{2+2r} 
 \] 
 of the minimal representation $\Pi[G]$ of $G$, and the trivial representation.  
 On the two summands the torus $T$  acts by the characters $|a|^{2+r}$ and $|a|^{2+2r}$, as indicated. 
We shall now rewrite the above filtration for the normalized Jacquet functor 
 $r_{B}(\Pi)= \delta_B^{-1/2} \cdot \Pi_U$.  
 
 \begin{prop}\label{P:filtration} 
 The  normalized Jacquet functor 
 $r_{B}(\Pi)= \delta_B^{-1/2} \cdot \Pi_U$, as a $T\times G$-module, has a filtration 
 \begin{itemize} 
 \item with a quotient  
 \[ 
r_{B}(\Pi_{\mathscr U}) =  \Pi[G]  \otimes |a|^{1+r} \oplus 1 \otimes |a|^{1+2r} 
\] 
where $t \in T$, given as the diagonal matrix $(a,a^{-1})$, acts as indicated, 
\item and a sub (normalized parabolic induction) 
\[ 
 \Ind_{P \times T}^{G \times T} (C_c(\GL_1))
\] 
where the action of $(t,m)\in T\times M$ on $f\in C_c(GL(1))$ is given by 
 \[ 
 f(\varpi(m)^{-1}a x ). 
 \] 
 
 \end{itemize}

 \end{prop} 
 \begin{proof} 
 The first bullet is clear. For the second, we need to rewrite the bottom $\delta^{-1/2}_B\cdot C_c(\Omega)$ as the induced representation. 
 Pick $v\in \Omega$ so that the stabilizer of the line through $v$ in $G$ is the maximal parabolic 
 $P=MN$.  In particular, $M$ acts on the line through $v$ by 
 the fundamental minuscule character $\varpi$.  Thus the stabilizer in $T\times G$ of $v$ is the set of 
 pairs $(t, mn)\in T \times P$ such that $a=\varpi(m)$. Thus $C_c(\Omega)$ is compactly induced (to $T\times G$) from the stabilizer. 
 Now the second bullet follows from induction in stages, first to $T\times P$, which gives us $C_c(\GL(1))$. 
 It remains to check normalizations for the parabolic induction. 
To that end, the modular character for $P$ is 
 \[ 
 \delta_{P}^{1/2}(m)=|\varpi(m)|^{1+2r}. 
 \] 
 The action of $t\in T$ on $f \in \delta^{-1/2}_B\cdot C_c(\Omega)$ is by $|a|^{1+2r} f(ax)$, which matches $\delta_{P}^{1/2}(m)$, as needed. 
 \end{proof}

 \medskip 
 
 \subsection{Computing the theta correspondence for principal series representations of $G$.} 
 \label{principal_series_SL}
 Recall that $\Ind_B^{\SL_2}(\chi)$ is irreducible unless $\chi=|\cdot|^{\pm 1}$ or $\chi$ is a non-trivial 
 quadratic character.  
 The induced representation $\Ind_B^{\SL_2}(|\cdot |)$ has the Steinberg representation $\St$ as a submodule, and the trivial representation  $1_{\SL_2}$ as a 
 quotient.  If $\chi$ is quadratic, but not trivial, then $\Ind_B^{\SL_2}(\chi)= \pi_{\chi}^{+}\oplus \pi_{\chi}^{-}$ is a direct sum of two irreducible 
 representations. If $\Ind_B^{\SL_2}(\chi)$ is irreducible, then $\Ind_{B}^{\SL_2}(\chi) \cong \Ind_{B}^{\SL_2}(\chi^{-1})$. Thus, 
 without loss of generality, if  $\Ind_B^{\SL_2}(\chi)$ is irreducible we can assume that $|\chi|$ is a positive power of $|\cdot |$.  
 
 \begin{prop}
 \label{prop_princ_series_SL2}  
 With the above conventions: 
 \begin{itemize} 
 \item If $\pi=\Ind_{B}^{\SL_2}(\chi)$ is irreducible, with $|\chi|$ is a positive power of $|\cdot |$,  then $\Theta(\pi)= \Ind_P^G(\chi)$. 
 \item $\Theta(1_{\SL_2})$ is a quotient of $ \Ind_P^G(|\cdot |)$ and $\Theta(\St)$ is a quotient of $ \Ind_P^G(|\cdot |^{-1})$. 
 \item If $\chi$ is non-trivial quadratic, then $\Theta(\pi_{\chi}^{\pm})$ are quotients of $ \Ind_P^G(\chi)$. 
 \end{itemize} 
  \end{prop} 
 \begin{proof} 
 We shall compute $\Hom_{\SL_2}(\Pi, \pi)$, the linear dual of $\Theta(\pi)$.  
 Assume $\pi$ is quotient of $\Ind_{B}^{\SL_2}(\chi)$, so it is a sub of $\Ind_{B}^{\SL_2}(\chi^{-1})$. Then 
 \[ 
  \Hom_{\SL_2}(\Pi, \pi) \subseteq  \Hom_{\SL_2}(\Pi, \Ind_{B}^{\SL_2}(\chi^{-1})), 
 \] 
  and we shall now compute the latter. By the Frobenius reciprocity, 
 \[ 
   \Hom_{\SL_2}(\Pi, \Ind_{B}^{\SL_2}(\chi^{-1})) \cong  \Hom_{T}(r_{B}(\Pi), \chi^{-1}).
  \] 
  Now we use the two step filtration of $r_{B}(\Pi)$ in Proposition \ref{P:filtration}. It gives a long exact sequence 
  \[ 
   \Hom_{T}( r_{B}(\Pi_{\mathscr U}), \chi^{-1}) \rightarrow  \Hom_{T}(r_{B}(\Pi), \chi^{-1})
  \rightarrow   \Hom_{T}(\Ind_{P\times T}^{G\times T} (C_c(\GL_1)), \chi^{-1}) \rightarrow \mathrm{Ext}_{T}( r_{B}( {\Pi_{\mathscr U}}), \chi^{-1})
\]  
 Since  $\chi^{-1}(a) \neq |a|^{1+r}, |a|^{1+2r}$, in each of the three bullets, the first and last term are 0. Hence 
 \[ 
 \Hom_{T}(r_{B}(\Pi), \chi^{-1})
  \cong \Hom_{T}( \Ind_{P\times T}^{G\times T} (C_c(\GL_1)), \chi^{-1}). 
 \] 
 Now observe that $(C_c(\GL_1)\otimes \chi)_T\cong \chi$, as $M$-modules, where $\chi$ is pulled to $M$ via $\varpi$.  Then \cite[Lemma 9.4]{GG}  gives 
 \[ 
 \Hom_{T}( \Ind_{P\times T}^{G\times T} (C_c(\GL_1)), \chi^{-1}) \cong \Hom_{\mathbb C}( \Ind_{P}^{G} (\chi), \mathbb C).
 \] 
  Hence $\Theta(\pi)$ is a quotient of  $\Ind_P^G(\chi)$, as claimed. 
 \end{proof} 
 
 In Section \ref{S:exceptional} we shall strengthen the above proposition and prove that $\Theta(1_{\SL_2})$ and $\Theta(\St)$ are both irreducible, and 
 that $\Theta(\pi_{\chi}^{\pm})$ are the two irreducible summands of $ \Ind_P^G(\chi)$ (the case $\chi$ is quadratic, non-trivial).

 \subsection{Computing the theta correspondence for principal series representations of $(G,V)=(\SL_2^3,V_2^{\otimes 3})$} 
 The statement is slightly different in this case, however the computation is similar as in the other three cases covered by 
  Proposition \ref{prop_princ_series_SL2}. 

  \begin{prop} \label{P:mini_principal} 
  If $\pi$ is an irreducible quotient of $\Ind_B^{\SL_2}(\chi)$, but not the Steinberg representation, then $\Theta(\pi)$ is a quotient of 
  $\Ind_B^{\SL_2}(\chi)\otimes \Ind_B^{\SL_2}(\chi)\otimes \Ind_B^{\SL_2}(\chi)$.  
   \end{prop} 
 \begin{proof} Recall that $B=TU$. Since $\pi$ is a submodule of $\Ind_B^{\SL_2}(\chi^{-1})$, we have 
 \[ 
 \Hom_{\SL_2}(\Pi, \pi) \subseteq \Hom_{\SL_2}(\Pi, \Ind_B^{\SL_2}(\chi^{-1})) \cong \Hom_T(r_B(\Pi), \chi^{-1}). 
 \] 
 Again, we need to describe $\Pi_U$.  Recall that $U$ is the center of the nilpotent radical of the $A_1^3$ maximal 
  parabolic subgroup with the nilpotent radical $\mathscr U$.  Thus we have an exact sequence 
  \[ 
  0\rightarrow C_c(\Omega) \rightarrow \Pi_U \rightarrow \Pi_{\mathscr U} \rightarrow 0 
  \] 
  where $\Omega$ is the minimal orbit in $V$, i.e. the highest weight orbit in the 8-dimensional representation $V_2^{\otimes 3}$. 
  The action of $T$ on $C_c(\Omega)$ is geometric, twisted by $\delta_B$. The action of $T$ 
  on $\Pi_{\mathscr U}$ is by $\delta_B.$ Indeed, using the usual realization of the $D_4$ root system, 
    $U$ is the root space of the highest root $\alpha=e_1+e_2$. We normalize so that the exponent of the trivial representation of $D_4$ is  
    $\rho=(3,2,1,0)$. The leading exponent of $\Pi$ is $\lambda =(2,1,1,0)$. It comes with multiplicity two. The other three exponents are 
    obtained by acting on $\lambda$ by three simple reflections corresponding to simple roots perpendicular to $\alpha$. Notice that this does not 
    affect evaluation on $\alpha^{\vee}$. Unnormalized exponents are obtained by subtracting $\rho;\;\lambda-\rho=(-1,-1,0,0)$. 
    Product with $\alpha^{\vee}$ is -2, i.e. $\alpha^{\vee}(a)$ acts as $|a|^2$, as claimed. Hence  $T$ acts on $r_B(\Pi_{\mathscr U})$ by $\delta_B^{1/2}$.  
  Since we assume that $\pi$ is not the Steinberg representation, then $\chi^{-1} \neq \delta_U^{1/2}$, and in this case 
  the top piece of the filtration can be ignored. Thus we have 
  \[ 
 \Hom_{\SL_2}(\Pi, \pi) \subseteq \Hom_T( \delta_B^{-1/2} C_c(\Omega) , \chi^{-1}). 
 \] 
 As in the proof of Proposition \ref{prop_princ_series_SL2} one sees that 
  \[ 
  \Hom_T( \delta_B^{-1/2} C_c(\Omega) , \chi^{-1}) \cong \Hom_{\mathbb C}(\Ind_B^{\SL_2}(\chi)\otimes \Ind_B^{\SL_2}(\chi)\otimes \Ind_B^{\SL_2}(\chi), \mathbb C). 
 \] 
    \end{proof}

\section{On cuspidal support of the lift}  \label{S:cuspidal_support} 

Let $P=MN$ be the maximal parabolic subgroup in $G$ as in the previous section. 
 Let $\sigma$ be an irreducible supercuspidal representation of $\SL_2$. In this section we shall prove that $\Theta(\sigma)\neq 0$ and 
 that $\Theta(\sigma)_N=0$. This gives us a separation between theta lifts of principal series representations and theta 
 lifts of supercuspidal representation, in terms of their cuspidal supports.

Let $\bar N$ be the radical of the opposite parabolic. Then $\bar N$ can be 
endowed a structure of a cubic Jordan algebra $J$  
(in particular, there is cubic norm $N_J$ and linear trace form $T_J$ on $J$ such that $T_J(1)=3$).
 Then $N\cong J^{\ast}$ as $M$-modules. 
With these identifications, if we fix an additive character $\psi $, an element $x\in J$ defines 
a character $\psi_x$ of $N$, via $\psi_x(n)= \psi(\langle n, x\rangle)$,  where $ \langle n, x \rangle$ is the natural pairing on $J^*\times J$. 

We shall now compute the space of 
twisted coinvariants $\Pi_{N,\psi_x}$ and derive some consequences.  
To that end, we shall make use of  a maximal parabolic $\mathscr P=\mathscr M\mathscr N$ in 
$\mathscr G$ containing $\SL_2 \times P$ such that $ \mathscr M^{\mathrm{der}}=\SL_2 \times M^{\mathrm{der}}$. The radical $\mathscr N$ has a filtration 
\[ 
\mathscr N_3 < \mathscr N_2 < \mathscr N_1 =\mathscr N 
\] 
such that the abelian subquotients, as $\SL_2 \times M^{\mathrm{der}}$-modules, are given by (see Sections 2-4 in  \cite{GS_E7})
\[ 
\begin{cases}  
\mathscr N_1/ \mathscr N_2 \cong V_2 \otimes J  \\
\mathscr N_2/  \mathscr N_3 \cong N \cong  J^{\ast} \\
\mathscr N_3 \cong V_2. 
\end{cases} 
\] 
 The filtration of $\mathscr N$ defines a filtration of the minimal representation 
\[ 
\Pi_3 \subseteq \Pi_2 \subseteq \Pi_1 \subseteq  \Pi 
\] 
 such that $\Pi/\Pi_i =\Pi_{\mathscr N_i}$ for all $i$. The first quotient $\Pi/\Pi_1=\Pi_{\mathscr N}$ is an explicitly known finite length $\mathscr M$-module. 
 For $i\geq 1$, $\Pi_i/\Pi_{i+1}$ is naturally a $\mathscr N_i/\mathscr N_{i+1}$-module, thus it is a sheaf over the Pontrjagin dual
$\widehat{\mathscr N_i/\mathscr N_{i+1}}$ (\cite{BZ_reps_GL}). By minimality of $\Pi$, the sheaf is supported on the minimal $\SL_2\times M$-orbit.  

 \begin{lemma}  Let $N_J$ denote the norm on $J$.  
 Then, for a particular identification $U\cong F$,  for every  $x\in J$ 
 \[ 
 (\Pi_3)_{N,\psi_x}  \cong \ind_U^{\SL_2} \psi_{N_J(x)} 
 \] 
 where $\psi_{N_J(x)} (u) = \psi(N_J(x) u)$ for  all $u\in U\cong F$. 
 \end{lemma} 
 \begin{proof}  Let $\psi_3$ be a nontrivial character of $\mathscr N_3$. We need to describe $\Pi_{\mathscr N_3,\psi_3}$, the fiber of the sheaf 
 corresponding to $\Pi_3$. 
  Since  $\SL_2$ acts transitively on non-trivial characters of $\mathscr N_3$, we can work with any one in particular. 
  So, if $v_1$ and $v_2$ is the  standard basis of $V_2$, that is, 
 $v_1$ is fixed by $U$ and $v_2$ fixed by $\bar U$, we can assume that $\psi_3$ is trivial on the line $\mathscr Z\subset \mathscr N_3$ spanned by $v_1$.

 We shall now describe a restricted root system of type $G_2$ on $\mathscr G$ that is eventually used to do needed computations. 
 The  normalizer of $\mathscr Z$ in $\mathscr G$ is a Heisenberg maximal parabolic subgroup $\mathscr Q=\mathscr L \mathscr H$. 
  Let $\mathscr B =\mathscr Q \cap\mathscr P$. It is a parabolic subgroup of co-rank 2. 
  Its two-dimensional central torus gives the restricted root system of type $G_2$, such that long root spaces are one dimensional and short root spaces are isomorphic to 
   $J$ or $J^*$, as modules for the Levi. The unipotent radical of $\mathscr B$ is $\mathscr U= U \mathscr N$. 
   It corresponds to a choice of positive roots in the $G_2$ root system. 
   If $\alpha$ is the long simple root, and $\beta$ the short simple root, then $U$ is the root group corresponding to $\alpha$, $v_2\otimes J$ is the root group 
   corresponding to $\beta$, $N\cong J^*$ is the root group corresponding to $\alpha+2\beta$ etc. Importance of this decomposition lies in the fact that the Lie bracket on 
${\mathrm {Lie}}( \mathscr U)$ can be described in terms of the Jordan algebra structure (\cite{GS_minimal},\cite{Rumelhart}).

   Observe that $\Pi_{\mathscr N_3, \psi_3}$ is naturally a module for $\mathscr U= U \mathscr N$. We shall now describe it. 
    Let $\chi$ be an extension of $\psi_3$ to $\mathscr H$. Then 
   $\Pi_{\mathscr H, \chi}\neq 0$ (and then it is one dimensional) if $\chi$ belongs to the highest weight orbit $\Omega$ in $\widehat{\mathscr H/\mathscr Z}$. 
    Each such  extension is uniquely determined by its restriction on $N$ \cite[Proposition 11.2]{GS_minimal}.  
    Thus let $\chi_x \in \Omega$ be the extension such that $\chi_x$ restricted to $N$ is $\psi_x$.  
   By the Frobenius reciprocity, we have a  homomorphism 
   \[ 
   \Pi_{\mathscr N_3,\psi_3} \rightarrow \ind_{\mathscr H}^{\mathscr U} (\chi_0) \cong C_c( v_2 \otimes J)\cong C_c(J). 
   \] 
   The composite of this map with the evaluation at $x\in J$ factors through $\Pi_{\mathscr H, \chi_x}$, which is one-dimensional. Thus the above homomorphism is 
   an isomorphism.  
   
Now it is not too difficult to explicate the action of $\mathscr U$ on $\Pi_{\mathscr N_3, \psi_3}$. It amounts to commuting elements in $v_2 \otimes J\cong J$ with 
elements in $\mathscr U$.  If we write elements in the root groups as $e_{\alpha}(u)$ where $u\in F$, $e_{\beta}(x)$ where $x\in J$, 
$e_{\alpha+ 2\beta}(y)$ where $y\in J^*$, etc, than the relevant group commutators in $\mathscr U$ are 
\[ 
[e_{\alpha +2\beta}(y), e_{\beta}(x)]=e_{\alpha+3\beta}(\pm \langle y, x \rangle) 
\] 
\[ 
[e_{\alpha}(u), e_{\beta}(x)]=e_{\alpha+\beta}(\pm ux) e_{\alpha+ 2\beta}(\pm u x^{\#}) e_{\alpha+3\beta}(\pm u N_J(x)) e_{2\alpha+3\beta}(\pm u^2 N_J(x))
\] 
where $x\mapsto x^{\#}$ is a quadratic map from $J$ to $J^*$ such that $\langle x^{\#}, x\rangle =3N_J(x)$ for all $x\in J$.  
Now it is clear that $y\in N\cong J^*$ and $u\in U\cong F$ (after adjusting these isomorphisms by signs, if necessary) 
act on $f\in C_c( J)$  by 
 \[ 
 (y\cdot f)(x)= \psi(\langle y, x\rangle)f(x) \text{ and } (u\cdot f)(x) = \psi(N_J(x) u) f(x), 
 \]
 respectively. 
 Now it is evident that $(\Pi_{\mathscr N_3,\psi_3})_{N,\psi_x}$ is one-dimensional and that $U$ acts on it via the character $\psi_{N_J(x)}$.  
 Thus we have described $(\Pi_3)_{N,\psi_x}$ as a sheaf on $\widehat{\mathscr N}_3$, supported away from 0. Lemma is now an easy consequence. 
 \end{proof} 
 
  Remark:  If $J=F$ with the norm $N_J(x)=x^3$ and the trace $T_J(1)=3$, which we use to identify $J$ and $J^*$, then $x^{\#}=x^2$ and the 
 above commutators reduces to well known commutators in the Chevalley group $G_2$.

 \begin{prop} \label{P:cuspidal_support} 
 Let $\sigma$ be an irreducible supercuspidal representation of $\SL_2$. Then $\Theta(\sigma)\neq 0$ and $\Theta(\sigma)_N=0$. 
 \end{prop} 
 \begin{proof} 
 If $N_J(x)\neq 0$ then $ \ind_U^{\SL_2} \psi_{N_J(x)}$ is the Gelfand-Graev representation. In this case it is trivial to see that other 
 subquotients of $\Pi$ do not contribute to $(N,\psi_x)$-coinvariants. Hence 
 $\Pi_{N,\psi_x}\cong  \ind_U^{\SL_2} \psi_{N_J(x)}$. Thus 
 \[ 
 \Theta(\sigma)_{N,\psi_x}\cong ( \ind_U^{\SL_2} \psi_{N_J(x)}\otimes \sigma^{\vee})_{\SL_2}. 
 \] 
 Since $\sigma$ is generic for some choice of Whittaker character, the above implies that 
 $\Theta(\sigma)\neq 0$. 
 
 For the second part, we shall prove that the space of $N$-coinvariants, for every subquotient $\Pi_i/\Pi_{i+1}$, does not contain $\sigma$. 
  We go from the bottom to top, starting with $\Pi_3$.  By the lemma and transitivity of induction, 
 \[ 
(\Pi_3)_N \cong \ind_U^{\SL_2} (1) \cong \ind_B^{\SL_2}(\ind_U^B (1)). 
\] 
Since $\SL_2/B$ is compact, the induction from $B$ to $\SL_2$ with compact support is the same as parabolic induction. 
Hence $\ind_U^{\SL_2} (1)$ is a parabolically induced module, so it cannot contain 
$\sigma$. 

For the second subquotient, observe that $N\cong \mathscr N_2 /  \mathscr N_3$. Since $\Pi_2/\Pi_3$ is a sheaf on a set of 
non-trivial characters of $\widehat{\mathscr N_2/\mathscr N_3}$, it follows that $(\Pi_2/\Pi_3)_{N}=0$.  For the third subquotient $\Pi_1/\Pi_2$, 
observe that the group $N\cong \mathscr N_2/  \mathscr N_3$ already acts trivially on it. 
This quotient is a sheaf over $\widehat{\mathscr N_1/\mathscr N_2} \cong V_2^*\otimes J^*$. 
By minimality of $\Pi$, that sheaf is supported on the  $\SL_2 \times M$-orbit  $\Omega$ consisting of highest weight vectors in $V_2^*\otimes J^*$, with 
one dimensional stalks. Hence 
\[ 
\Pi_1/\Pi_2\cong C_c(\Omega) 
\] 
where  the action of $\SL_2 \times M$ on $C_c(\Omega)$ is geometric. 
   Assume that $\sigma$ appears as a subquotient. 
   Since $\sigma$ is supercuspidal, it is also a submodule of $C_c(\Omega)$. Let $\omega\subset \Omega$ be any $\SL_2$-orbit. 
   Consider the composite 
   \[ 
   \sigma \rightarrow C_c(\Omega)\rightarrow C_c(\omega). 
   \] 
   Clearly there exists an $\SL_2$-orbit $\omega$ such that the composite map is non-trivial. 
   Since elements in $\Omega$ are highest weight vectors, there exists point in $\omega$ whose stabilizer in $\SL_2$ is the (usual) unipotent group $U$. 
   Thus we have, by transitivity of induction, 
   \[ 
   C_c(\omega)\cong \ind_U^{\SL_2}(1)\cong  \ind_B^{\SL_2}(\ind_U^B(1)). 
   \] 
Again, a contradiction. It remains to deal with the top piece, $\Pi_{\mathscr N}$. This is a finite length $\mathscr M$-module generated by Iwahori 
fixed vectors.  Thus $\sigma$ does not appear here, either.

 \end{proof}

 \section{Classical theta for $\SL_2 \times \Spin_{2n}$}  \label{S:classical} 
 
 In this section we derive some results for classical theta correspondences that we shall need later. 

 We want to describe the classical theta correspondence for the pair $\SL_2 \times \Spin_{2n}.$ To that end, we employ the Fourier-Jacobi functor (\cite{Berndt_Schmidt_Jacobi},\cite{Weissman_FJ_small}).
 If there are several FJ--functors around -with respect to different groups, we add a subscript.  In our setting, FJ-functor is an exact functor which sends a smooth representation of a group to a smooth representation of the next group in the appropriate FJ-tower (\cite{HS_Jordan_Eisenstein}). One of the most important properties is that it annuls only the trivial representation (e.g.~ \cite{Weissman_FJ_small} ). Assume that  $2n\geq 6$, so $\Spin_{2n}$ has a Heisenberg parabolic with Levi subgroup isomorphic to $\SL_2\times \Spin_{2n-4}.$  
Let $FJ_{\Spin}$ be the Fourier-Jacobi with respect to $\Spin_{2n}$. The output is a representation of $\SL_2\times \Spin_{2n-4}$. 

\begin{lemma}
\label{lem_FJ_of_weil_4n}
The $FJ_{\Spin}$ functor, applied to Weil representation for the dual pair $\SL_2 \times \Spin_{2n}$ for $n\ge 3,$ 
gives the regular representation of $\SL_2$ and the action of $\Spin_{2n-4}$ is trivial.
\end{lemma}

\begin{proof}
Let $\omega_{\psi}$ be the Weil representation   for the dual pair $\SL_2 \times \Spin_{2n}.$ As before, we denote the underlying irreducible representation of the corresponding Heisenberg group in the same way. Sometimes, if there are several Heisenberg groups around, we add a dimension of the corresponding symplectic space to the notation, i.e. $\omega_{\psi}^{2n}.$  Note that, by \cite[Proposition 2]{KS_minimal}, we have
\[\SL_2 \times \Spin_{2n}\to \syp(4n),\]
since, as explained, this group can be viewed as a Levi subgroup of the Heisenberg parabolic subgroup of $\Spin_{2n+4}.$ We can use explicit formula for the action of the dual pair $\SL_2 \times \rO_{2n},$ as given in \cite{Kudla1}, but using "the opposite" polarisation to one there. Namely, if $\SL_2$ acts on the symplectic space $W$ and $V=X\oplus Y$ is a complete polarisation of the split quadratic space of dimension $2n,$ we realize $\omega_{\psi}$ on the Schwarz space $S(X\otimes W)\cong S(W^n).$ We want to determine $(\omega_{\psi})_{U,\psi},$ where, as previously, $U$ is the root subgroup of the highest root  $\beta$ of $\Spin_{2n}.$ Thus, $(\omega_{\psi})_{U,\psi}$ is $SL_2\times SL_2 \times \Spin_{2n-4}$--module. To calculate it, we can consider $U$ as a root space belonging to $O_{2n}$-this makes sense after taking the group of $F$-points and noting that  $\Spin_{2n}$ acts on $W\otimes V$ by conjugation-as a part of the appropriate Levi subgroup.

\smallskip
Since $U$ belongs to the unipotent radical of the Siegel parabolic subgroup, in $(\omega_{\psi},S(W^n))$ it acts by homotety, cf. Chapter II, Section 4 of \cite{Kudla1}.  In that case, one can compute $(\omega_{\psi})_{U,\psi},$ as in Lemma 2.2 of \cite{MS} to get
\[S(W^n)_{U,\psi}\cong S(W^{n-2}\times \SL_2)\cong S(W^{n-2})\otimes S(\SL_2). \]
Note that the action of the first $SL_2$ (a member of a dual pair) in this realization is purely geometrical (by the left translation), and the second $SL_2$ (part of the Levi subgroup of $\Spin_{2n}$) also acts geometrically-by the right translation,  as a part of the corresponding  Siegel parabolic subgroup (again Chapter II, Section 4 of \cite{Kudla1}). This means that we have have the regular representation of $\SL_2\times \SL_2$ on $S(\SL_2).$  On the other hand, on $S(W^{n-2})$ we have the usual Weil representation of a dual pair $\SL_2\times \rO_{2n-4},$ i.e.   $\SL_2\times \Spin_{2n-4},$ as can be seen from the appropriate explicit formulas (we just removed two  coordinate hyperbolic planes from the quadratic space).  So, $S(W^{n-2})$ carries an irreducible representation $\omega_{\psi}^{4n-8}.$ The group $\rO_{2n-4}$ acts on $S(\SL_2)$ by a character given in Proposition 4.3 of \cite{Kudla1}, so $\Spin_{2n-4}$ acts trivially on it. Let  $\mathscr U$ be the Heisenberg parabolic subgroup of $\Spin_{2n}.$  We have
\begin{align*}
&FJ(\omega_{\psi})=FJ(\omega_{\psi}^{4n})=\Hom_{\mathscr U}(\omega_{\psi}^{4n-8}, (\omega_{\psi}^{4n})_{U,\psi})=\Hom_{\mathscr U}(\omega_{\psi}^{4n-8}, S(W^{n-2})\otimes S(\SL_2))\cong\\
&\Hom_{\bbC}(\omega_{\psi}^{4n-8},\omega_{\psi}^{4n-8})\otimes S(\SL_2).
\end{align*}
\end{proof}

\begin{cor}
\label{cor_FJSpin2n}
 If $\sigma$ is an irreducible representation of 
$\SL_2$ and $\Theta_{\psi}(\sigma)$ the big theta lift to $\Spin_{2n}$, then $FJ_{\Spin_{2n}}(\Theta_{\psi}(\sigma))$ is an irreducible representation of $\SL_2\times\Spin_{2n-4}$, 
$\sigma$ on $\SL_2$, and trivial on the $\Spin_{2n-4}$-- factor. 
 \end{cor}

\vskip 15pt 

We  are  able to now completely describe the theta correspondence in question. In order to do that, we note the following. The group $\Spin_{2n}$ has a standard parabolic subgroup $P=MN$ with $N$ commutative and $M$ of semisimple type $D_{n-1}.$ There is a rational character $\omega$ of $M/[M,M]$ described in the second section of \cite{HS_Jordan_Eisenstein}; we can compose it with the  character $\chi_u|\cdot|^s,$ where $\chi_u$ is a unitary character of  $F^*$ and $s$ is real.  Let 
\[ 
I(s,\chi_u)=\Ind_{P}^{\Spin_{2n}}(\chi_u|\cdot|^s\circ \omega), 
\] 
the normalized induction.  Note that $I(s,\chi_u)$ is reducible if and only if $\chi_u^2=1,\chi_u\neq 1$ and $s=0$ or $\chi_u=1$ and $s\in\{\pm 1,\pm (n-1)\};$ in the case of reducibility, the length of the representation is two (cf.~Theorem 3.2 of \cite{HS_Jordan_Eisenstein}). If $\chi_u=1$ the unique irreducible quotient of $I(1,1)$ is the minimal representation and of $I(n-1,1)$ the trivial representation. We denote $J(s,\chi_u)=\Ind_B^{\SL_2}\chi_u|\cdot|^s.$

\begin{lemma}
\label{lem_classical_lift_general_case}
Let $\sigma$ be an irreducible representation of $\SL_2.$ Then its full  lift $\Theta_{\psi}(\sigma)$ to $\Spin_{2n}$ is irreducible except when $\sigma=J(n-1,1).$ In more detail, the following holds:
\begin{enumerate}
\item $\Theta_{\psi}(J(s,\chi_u))=I(s,\chi_u)$
\item $\Theta_{\psi}(1_{\SL_2})$ is the unique  irreducible quotient of $I(1,1),$ i.e.~the minimal representation of $\Spin_{2n}.$
\item $\Theta_{\psi}(\St_{\SL_2})$ is the unique irreducible subrepresentation of $I(1,1).$
\item  $\Theta_{\psi}(\pi_{\chi}^+)\oplus \Theta_{\psi}(\pi_{\chi}^-)=I(0,\chi_u);$ here
$\pi_{\chi}^{\pm}$ are irreducible tempered representations of $\SL_2$ described in  
Section \ref{principal_series_SL}.
\item If $\sigma$ is supercuspidal, the full lift is an  irreducible representation whose Langlands data is easier to explain when viewed as a lift to $\rO_{2n},$ cf.~Theorem 6.7 of \cite{BH_theta}.
	\end{enumerate}
If $\pi$ is an irreducible representation of $\Spin_{2n}$ which has a non-zero lift to $\SL_2$--the list of those follows from the previous claims in this lemma-then its full lift $\Theta_{\psi}(\pi)$ is an irreducible representation of $\SL_2.$	
\end{lemma}

\begin{proof}First we cover the case of the principal series representations of $\SL_2.$
 
\smallskip

First assume that $\sigma=J(s,\chi_u)$ is irreducible. Then, $\Theta_{\psi}(\sigma)=I(s,\chi_u).$ Indeed, using standard techniques of Kudla's filtration (e.g.~Corollary 5.3. of \cite{Atobe_Gan_theta_Temp}) together with Corollary \ref{cor_FJSpin2n}, we get that 
$\Theta_{\psi}(\sigma)=\Theta_{\psi}(J(s,\chi_u))$ is a non-zero quotient of $I(s,\chi_u)$ and the latter representation is irreducible in all the cases in which $J(s,\chi_u)$ is irreducible except in the case $s=\pm (n-1)$ and $\chi_u=1.$ Then, we get that the small theta lift of $J(n-1,1)$ is the trivial representation of $\Spin_{2n}.$ But, in that case, if we would have $\Theta_{\psi}(J(n-1,1))=\theta_{\psi}(J(n-1,\chi_u))=1_{\Spin_{2n}},$ we would get contradiction with Corollary \ref{cor_FJSpin2n}, because $FJ_{\Spin_{2n}}(1_{\Spin_{2n}})=0.$ Thus,  $\Theta_{\psi}(J(n-1,1))=I(n-1,1).$

\smallskip

The relation $\Theta_{\psi}(\pi_{\chi}^+)\oplus \Theta_{\psi}(\pi_{\chi}^-)=I(0,\chi_u)$  follows immediately as in the previous discussion.

\smallskip

Analogously, if $\sigma=1_{\SL_2},$ its  full theta lift to $\rO_{2n}$ is an irreducible representation (the Kepler representation), which is not a character, so that  the lift $\Theta_{\psi}(1_{\SL_2})$ to $\Spin_{2n}$ does not have the trivial subquotient. Then, Corollary \ref{cor_FJSpin2n} guarantees that $\Theta_{\psi}(1_{\SL_2})$ is an irreducible quotient of $I(1,1),$ thus the minimal representation. In the same way, $\Theta_{\psi}(\St_{\SL_2})$ to $\Spin_{2n}$ is irreducible, because its full lift to $\rO_{2n}$ does not have a subquotient which is a character, because of its cuspidal support (cf.~\cite{Muic}), so that, again, the lift $\Theta_{\psi}(\St_{\SL_2})$ to $\Spin_{2n}$ does not have the trivial subquotient. This means that $\Theta_{\psi}(\St_{\SL_2})$ is the unique subrepresentation of $I(1,1).$
{}
\smallskip

If $\sigma$ is a cuspidal representation, its full lift to $\rO_{2n}$ is irreducible and is not a character (we have an explicit description of it in terms of the Langlands data, cf.~\cite{BH_theta}); thus $\Theta_{\psi}(\sigma)$ does not have the trivial subquotient and so is irreducible.

\smallskip
 Now, we want  to prove the opposite direction; namely, if $\pi$ is an irreducible representation of $\Spin_{2n}$ such that $\Theta_{\psi}(\pi),$ its theta lift to $\SL_2$ is non-zero, then $\Theta_{\psi}(\pi)$ is irreducible. 

 First, assume that $\pi\neq 1_{\Spin_{2n}}$ is an irreducible representation of $\Spin_{2n}$ which appears in theta correspondence with $\SL_2$ so that
 \[\omega_{\psi}^{4n}\twoheadrightarrow \pi\otimes \Theta_{\psi}(\pi),\]
 where $\Theta_{\psi}(\pi)$ is a non-zero representation of $\SL_2.$ Since $\Theta_{\psi}(\pi)$ is of finite length, there exists an irreducible quotient   $\xi$, so that, by Lemma \ref{lem_FJ_of_weil_4n},
 \[S(SL_2)\twoheadrightarrow FJ_{\Spin_{2n}}(\pi)\otimes \Theta_{\psi}(\pi)\twoheadrightarrow FJ_{\Spin_{2n}}(\pi)\otimes \xi;\]
 this means $FJ_{\Spin_{2n}}(\pi)=\xi$ and, in, turn, $\Theta_{\psi}(\pi)=\xi.$ Thus, $\Theta_{\psi}(\pi)$ is irreducible.
 
 \smallskip
 If $\pi=1_{\Spin_{2n}}$, then we know that it occurs in the theta correspondence with $\SL_2,$ $\theta_{\psi}(1_{\Spin_{2n}})=J(n-1,1).$ On the other hand, every other subquotient of $\Theta_{\psi}(1_{\Spin_{2n}})$ would have the same cuspidal support as 
$\theta_{\psi}(1_{\Spin_{2n}}),$ but $J(n-1,1)$ is irreducible; thus  $\Theta_{\psi}(1_{\Spin_{2n}})=J(n-1,1).$

\end{proof}

\bigskip

\section{Exceptional correspondence}  \label{S:exceptional} 

We return to the situation where $\Pi$ is the minimal representation of $\mathscr G$, and study its restriction to the dual pair $\SL_2\times G$
using the Fourier-Jacobi functor. In particular, is $\sigma$ is an irreducible supercuspidal representation of $\SL_2$, 
we shall prove that $\Theta(\sigma)$ is non-zero and irreducible.  

\vskip 5pt

   Let $Q=LH$ be the Heisenberg parabolic subgroup of $G$. The cases are given by the following table: 
  \[ 
 \begin{array}{c||c|c|c} 
 \mathscr G &  E_6 & E_7 & E_8  \\ \hline 
G &  {\SL_6} & \Spin_{12} & E_7  \\ \hline 
L^{\mathrm {der}} &  \SL_4 \cong \Spin_{6} & \SL_2' \times \Spin_{8}&  \Spin_{12}  \\ 
 \end{array}. 
 \]

  Let $Z$ be the center of $H$, and fix an additive character $\psi$ of $Z$. Let $\omega_{\psi}$ be the corresponding Heisenberg representation of $H$. If  
  $\pi$ is a smooth representation of $G$, then $\pi_{Z,\psi}$ is a multiple of $\omega_{\psi}$, in fact, there exists a smooth $L^{\mathrm {der}}$-module 
  $FJ_G(\pi)$ (\cite{Weissman_FJ_small}) such that 
  \[ 
  \pi_{Z,\psi} \cong  FJ_G(\pi)\otimes \omega_{\psi}.
  \] 
    We shall now compute $FJ_G(\Pi)$. Observe that it is a module for 
  $\SL_2 \times L^{\mathrm {der}}$, where $\SL_2$ is the centralizer of $G$ in $\mathscr G$.  To that end, 
  observe  that $Z$ is a root space in $G$  and $Q$ is the normalizer of $Z$ in $G$. Similarly, 
the normalizer of $Z$ in $\mathscr G$ is a Heisenberg parabolic denoted (abusing the notation) by $\mathscr Q =\mathscr L \mathscr U$.  
 Recall that $\Pi_{Z,\psi}\cong \Omega_{\psi}$, the Heisenberg representation of $\mathscr U$. 
One easily checks that the complement of the symplectic space $H/Z$ in $\mathscr U/Z$ is the symplectic space 
\[ 
V_2 \otimes V_{2n} 
\] 
for $2n=6,8$ and $12$, respectively, where $V_2$ is the standard representation of $\SL_2$ (the centralizer of $G$), and $V_{2n}$ is the orthogonal 
representation of $\Spin_{2n}$.  In the case  $2n=8$, the factor $\SL_2'$ acts trivially!  Thus we have a factorization  
\[ 
\Omega_{\psi} = \omega_{\psi}^{4n} \otimes \omega_{\psi}. 
\] 
The following proposition is an easy consequence:  

\begin{prop} Let $\Pi$ be the minimal representation of $\mathscr G$. Then 
\[ 
FJ_G(\Pi) \cong  \omega_{\psi}^{4n} 
\] 
the Weil representation of $\SL_2 \times \Spin_{2n}$, $2n=6,8,12$, respectively,  with $\SL_2'$ acting trivially in the case $2n=8$.  
\end{prop} 

\begin{cor} \label{C:fj} 
 Let $\sigma$ be an irreducible representation of $\SL_2$. Let $\Theta(\sigma)$ be the big theta lift to $G$. 
 Then 
 \[ 
 FJ_G(\Theta(\sigma)) \cong \Theta_{\psi}(\sigma),
 \] 
  where $\Theta_{\psi}(\sigma)$ is the classical theta lift of  $\sigma$ to 
$\Spin_{2n}$, for $2n=6,8,12$, respectively. 
\end{cor} 
\begin{proof} It follows from the following sequence of isomorphisms 
\[ 
FJ_G(\Theta(\sigma)) =FJ_G((\Pi\otimes \sigma^{\vee})_{\SL_2}) \cong (FJ_G(\Pi)\otimes \sigma^{\vee})_{\SL_2} \cong  \Theta_{\psi}(\sigma). 
\] 
\end{proof} 

The above corollary allows us to prove the following set of results: 

\begin{prop}
\label{prop_lift_cusp_irr}
 Let $\sigma$ be an irreducible supercuspidal representation of $\SL_2$. Then $\Theta(\sigma)$ is non-zero and irreducible.
Moreover, $\Theta(\sigma) \cong \Theta(\sigma')$ only if $\sigma'\cong\sigma$. 
\end{prop} 
\begin{proof} 
By Proposition \ref{P:cuspidal_support} we know that $\Theta(\sigma)$ does not contain the trivial representation of $G$  as a subquotient. 
Since $FJ_G$ kills only the trivial representation,  Corollary \ref{C:fj}  implies that the length of $\Theta(\sigma)$ is no more than that of  $\Theta_{\psi}(\sigma)$. 
Since, by Lemma \ref{lem_classical_lift_general_case}, $\Theta_{\psi}(\sigma)$ is irreducible we get that $\Theta(\sigma)$ is irreducible
 and non-isomorphic to $\Theta(\sigma')$ unless $\sigma\cong\sigma'$. 
\end{proof}  
 
 Next, we finish the description 
 of $\Theta(\sigma)$ for principal series representations started in Proposition \ref{prop_princ_series_SL2}.  
 
 \begin{prop} Let $P$ be the maximal parabolic in $G$ as in Proposition \ref{prop_princ_series_SL2}.  Then 
 \begin{itemize} 
 \item $\Theta(1_{\SL_2})$ is the unique irreducible quotient of $\Ind_P^G|\cdot|$,
 \item $\Theta(\St_{\SL_2})$ is the unique irreducible quotient of $\Ind_P^G|\cdot|^{-1}$,
 \item If $\chi$ is quadratic non-trivial, then 
 $\Theta(\pi_{\chi}^{\pm})$ are the two irreducible summands of $ \Ind_P^G(\chi)$. 
 \end{itemize} 
  \end{prop} 
  \begin{proof} For the third bullet, we know that $\Theta(\pi^{\pm}_{\chi})$ are contained in $\Ind_P^G(\chi)$. 
  Since $\Ind_P^G(\chi)$ does not have a trivial subqotient (it follows from Proposition \ref{P:degenerate_ps}) the third 
  bullet follows from Corollary \ref{C:fj} , using that $\Theta_{\psi}(\pi^{\pm}_{\chi})$  are irreducible representations of $\Spin_{2n}$ 
  by Lemma \ref{lem_classical_lift_general_case}.  
 The first two bullets are proved in  the same way. 
   \end{proof} 
 
 Putting things together we have the following: 
\begin{thm} Let $\sigma$ be an irreducible representation of $\SL_2$. Then $\Theta(\sigma)$ is a non-zero, finite length representation of $G$ with unique irreducible 
quotient denoted $\theta(\sigma)$. If $\theta(\sigma)\cong \theta(\sigma')$, for two  irreducible representations $\sigma$ and $\sigma'$, then $\sigma\cong\sigma'$.  
\end{thm}

 \section{Lift from $G$ to $\SL_2$} \label{S:theta_back}  
  
  Recall that we have the dual pair $\SL_2\times G$.  Let $\pi$ be an irreducible representation of $G$.  We can define its 
  big-theta lift to $\SL_2$ by the usual $\Theta(\pi) = (\Pi \otimes \pi^{\vee})_G$.  If $\Theta(\pi)\neq 0$, a priori it is not clear 
  that it has an irreducible quotient. Our first result is that $\Theta(\pi)$ has finite length. This clearly implies that  it has  
  an irreducible $\SL_2$-quotient $\sigma$.  In turn, this implies that $\pi$ is the lift of $\sigma$. We shall then show that $\Theta(\pi)$ is 
  always irreducible (if non-zero) and then use this to prove that $\Omega_{\psi}$ is multiplicity free $G$-module.  
 
 \smallskip 
  
  \begin{prop} \label{P:finite} 
  Let $\pi$ be an irreducible representation of $G$. Then $\Theta(\pi)$ is a finite length $\SL_2$-module.  
\end{prop} 
\begin{proof}  
If $\sigma$ is a supercuspidal summand of $\Theta(\pi)$, then $\pi\otimes \sigma$ is a quotient of the minimal representation, 
and by Proposition \ref{prop_lift_cusp_irr}, $\pi=\Theta(\sigma)$. Since lifting from $\SL_2$ to $G$ is one to one, 
$\sigma$ is the unique supercuspidal summand of $\Theta(\pi)$.  Now, to finish the proof, it suffices to show that $r_{B}(\Theta(\pi))$ is finite 
dimensional, where $B$ is the Borel subgroup of $\SL_2$. 
 To that end, observe that $r_{B} (\Theta(\pi))\otimes \pi$ is a quotient of $r_{B} (\Pi)$, hence it suffices to show that 
\[ 
\Hom_G(r_{B} (\Pi), \pi) 
\] 
is finite dimensional. Now recall, from Proposition \ref{P:filtration}, the filtration of {$r_{B} (\Pi)$,} where the top is a finite length $G$-module, 
a direct sum of  trivial and minimal representation, while the large bottom piece is 
$\Ind_{P}^{G} (C_c(\GL_1))$.  Thus we need to prove finiteness of 
\[ 
\Hom_G( \Ind_{P}^{G} (C_c(\GL_1)), \pi) \cong   \Hom_M(C_c(\GL_1), r_{\bar P}(\pi)). 
\] 
Recall that $M$ acts through the weight $\varpi$ on $C_c(\GL_1)$ so the dimension of the right hand side is bounded by the number of 
one-dimensional $M$-subquotients. But this is finite since  $r_{\bar P}(\pi)$ has a finite length, as an $M$-module.   
\end{proof} 

 We remark, if  $\pi=1_G$ is the trivial representation of  $G$, then the above argument gives 
\[ 
\dim \Hom_G(r_{B} (\Pi), 1_G) \leq 2. 
\]

\begin{cor} 
Let $\pi$ be an irreducible representation of $G$. Then $\Theta(\pi)\neq 0$ if and only if $\pi$ is a lift of an irreducible representation $\sigma$ of $\SL_2$.  
\end{cor} 
\begin{proof} 
By Proposition \ref{P:finite}, $\Theta(\pi)$ has finite length. Hence it has an irreducible quotient $\sigma$. Thus $\pi$ is a lift of $\sigma$. 
\end{proof} 

\begin{prop} Let $\pi$ be a lift of  an irreducible representation $\sigma$ of $\SL_2$.  Then $\Theta(\pi)\cong \sigma$.  
\end{prop}
\begin{proof}  
 If $\pi=1_G$ we proved above that $r_{B} (\Theta(1_G))$ is at most 2-dimensional. On the other hand, $1_G$ is the 
lift of an irreducible principal series representation  $\sigma$ of  $\SL_2$, hence $\Theta(1_G)\cong \sigma$ as there is no room for anything else. 
For the general case we need the following.

\begin{lemma} Let $\pi$ the unique irreducible quotient of $\Theta(\sigma)$.  If $\pi \neq 1_G$  then $FJ_G(\pi)$ is the unique irreducible 
quotient of  the classical theta lift $\Theta_{\psi}(\sigma)$ to $\Spin_{2n}$, for $2n=6,8$ and $12$, respectively.  
\end{lemma} 
\begin{proof}  Since $\pi\neq 1_G$,  $\Theta(\sigma)=\pi$ unless $\pi$ is the minimal representation of $G$. Recall that 
$FJ_G(\Theta(\sigma))$ is $\Theta_{\psi}(\sigma)$. This, in turn, is irreducible unless it has the trivial representation as the unique quotient.  
This happens precisely when $\pi$ is the minimal representation when $FJ(\pi)$ is the trivial representation. This completes the proof. 
\end{proof}

Assume $\pi\neq 1_G$, so $FJ(\pi)\cong \theta_{\psi}(\sigma)$, the unique irreducible quotient of $\Theta_{\psi}(\sigma)$. We have  a surjection
\[ 
\omega_{\psi}^{4n} \cong FJ_{G}(\Pi) \rightarrow  \Theta(\pi) \otimes FJ_{G}(\pi) \cong   \Theta(\pi) \otimes \theta_{\psi}(\sigma). 
\] 
 But, according to Lemma \ref{lem_classical_lift_general_case}, the lift from $\Spin_{2n}$ to $\SL_2$ is necessarily irreducible, so $\sigma=\Theta(\pi).$
 
\end{proof}

Thus now we can analyze the Weil representation $\Omega_{\psi} \cong \Pi_{U,\psi}$ via the theta correspondence. 

\begin{thm} \label{T:mult_free} 
 Let $\pi$ be an irreducible representation of $G$. Then $\dim \Hom_G(\Omega_{\psi},\pi) \leq 1$ and it is 1 precisely when $\pi$ is a 
 lift of an irreducible $(U,\psi)$-generic representation of $\SL_2$. 
\end{thm} 
\begin{proof} We have 
\[ 
(\Omega_{\psi}\otimes \pi^{\vee})_G \cong  (\Pi_{U,\psi}\otimes \pi^{\vee})_G\cong \Theta(\pi)_{U,\psi} .  
\] 
Since $\Theta(\pi) \neq 0$ if and only if $\pi$ is a lift of an irreducible representation $\sigma$ of $\SL_2$, and then $\Theta(\pi)=\sigma$, and the theorem follows. 
\end{proof}

 \section{Mini theta}  \label{S:mini_theta}
 
 Let $V$ be the standard (symplectic) 2-dimensional representation of $\SL_2(F)$, realized as the set of $2\times 1$ column vectors. 
 Now we consider the $n$-fold tensor product $V\otimes  \cdots \otimes V$. This a faithful representation of the group 
 \[ 
S_n= \SL_2^n/\mu_2^{n-1}
 \] 
 where $\mu_2^{n-1}$ is the subgroup of the center of $\SL_2^n$, consisting of $(x_1, \ldots, x_n)$ such that $x_1 \cdots x_n=1$.  
 In order to understand the corresponding group of $F$-points, observe that on $V\otimes  \cdots \otimes V$ we have a faithful action of 
 \[ 
 G_n=  \GL_2^n/{\mathbb G}_m^{n-1} 
 \] 
 where ${\mathbb G}_m^{n-1}$ is the subgroup of the center of $\GL_2^n$ consisting of  $n$-tuples $(z_1, \ldots, z_n)$ of scalar matrices such that 
 $z_1\cdots z_n=1$.  The group $G_n(F)$ is given by the same quotient, by Hilbert 90. If $\pi$ is an irreducible representation of $\GL_2(F)$ then 
 $\pi\otimes \ldots \otimes \pi$ (tensor $n$-times) is naturally a representation of $G_n(F)$.

 \vskip 5pt

  Let $\det$ be the character of $G_n$ that restricts to the usual determinant on each $\GL_2$. Then $S_n$ is the kernel of the determinant, this 
  is an easy check over algebraic closure of $F$.  This description gives as the following inductive construction of $S_n$, which will play an important role. 
  Let 
  \[ 
  (G_{n-1} \times \GL_2)^{\det}=\{ (g,h) ~|~ \det g \cdot \det h=1\}. 
  \] 
  Let $\Delta {\mathbb G}_m$ be embedding of ${\mathbb G}_m$ in the center of the above group given by 
  $z\mapsto ((z, 1, \ldots, 1),z^{-1})= (1, z, \ldots, 1),z^{-1}) \ldots $.  Then 
  \[  
  S_n \cong (G_{n-1} \times \GL_2)^{\det}/ \Delta {\mathbb G}_m
  \] 
  and this is true on the level of $F$-points.   Thus any smooth $S_n$-module $\Omega$ can be viewed as an $ (G_{n-1} \times \GL_2)^{\det}$-module. Let 
  $\pi$ be an irreducible representation of $\GL_2(F)$. Let $\SL_2 \subset (G_{n-1} \times \GL_2)^{\det}$, sitting in the second factor. 
  Observe that $ (G_{n-1} \times \GL_2)^{\det}/ \SL_2 \cong G_{n-1}$. Thus 
  \[ 
 \Theta(\pi)=  (\Omega\otimes \pi^{\vee})_{\SL_2} 
  \] 
  is naturally a smooth $G_{n-1}$-module.  In the following three subsections, for $n=2,3,4$ we shall exhibit $\Omega$ such that 
  $\Theta(\pi)=\pi\otimes \ldots \otimes \pi$ for all irreducible Whittaker generic $\pi$.

 \subsection{Two $\SL_2$} 
 In this case $S_2= \SL_2\times \SL_2/\mu_2 \cong \SO(4)$. Indeed, 
  $V\otimes V$ is an orthogonal representation of $\SL_2\times \SL_2/\mu_2$. It can be identified with the space $M_2$  of $2\times 2$ matrices 
 by the map 
 \[ 
 v\otimes w \mapsto v\cdot w^{\top} 
 \] 
 where $w^{\top}$ is the transpose of $w$. The determinant is a quadratic form that turns $M_2$ into a quadratic space. 
  The action of $(g,h) \in (\GL_2\times \GL_2)^{\det} $ on $x\in M_2$ is given by $gxh^{\top}$. This gives 
 as isomorphism $ (\GL_2\times \GL_2)^{\det} / \Delta {\mathbb G}_m\cong \SO(4)$.  Since the  action preserves the determinant, it 
 follows that $ (\GL_2\times \GL_2)^{\det}$ acts naturally on the regular representation $C_c(\SL_2(F))$.  Now we have 
 
 \begin{prop} \label{P:theta_2} 
 Let $\pi$ be an irreducible representation of $\GL_2(F)$. Then 
 \[ 
 (C_c(\SL_2(F))\otimes \pi^{\vee})_{\SL_2} \cong \pi
 \] 
 as $(\GL_2\times \GL_2)^{\det}/\SL_2\cong \GL_2$-modules. 
 \end{prop} 
  \begin{proof} We have a natural $\GL_2$-intertwining map 
  \[ 
  (C_c(\SL_2) \otimes \pi^{\vee})_{\SL_2} \rightarrow \pi 
  \] 
  given by $(f,v) \mapsto \pi(f)v$.  To prove that this is bijective, restrict $\pi=\pi_1 + \cdots + \pi_s$ to $\SL_2$, and use 
  \[ 
  (C_c(\SL_2) \otimes \pi^{\vee}_i)_{\SL_2} \cong  \pi_i
  \] 
  for any irreducible representation $\pi_i$ of $\SL_2$.  
  
  \end{proof}

 \subsection{Three $\SL_2$} Observe that $V\otimes V\otimes V$ is naturally a symplectic space. Hence 
 \[ 
 \SL_2 \times \SL_2 \times \SL_2 /\mu_2^2 \subset \Sp_8
 \] 
 Fix an additive character  $\psi$ of $F$. It gives us 
 a Whittaker datum $(U,\psi)$ for $\SL_2$, as well as a Weil representation $\omega_{\psi}$ of $\Sp_8$. By restriction, $\omega_{\psi}$ is an $S_3$-module. 
 Decomposing $\omega_{\psi}$ under the action of $S_3$ is closely related to the classical theta correspondence, where 
(any) two $\SL_2$ are the factors of $\SO_4 \cong (\GL_2 \times \GL_2)^{\det}/\Delta({\mathbb G}_m)$  and the 
 remaining $\SL_2$ is a symplectic group.

 We realize $\omega_{\psi}$, as a Schroedinger model, on the space of compactly supported smooth functions $C_c(M_2(F))$, 
 with the geometric action of $(\GL_2 \times \GL_2)^{\det}$, while the action of the third  $\SL_2$ is generated by 
 \[ 
 u\cdot f(x) = \psi(u\det(x))f(x), 
 \] 
 for $u\in U$,  and a Fourier transform. 
 Namely, if $w$ is the non-trivial Weyl group element of the third $\SL_2$ and $\phi \in C_c(M_2(F))$, then, by \cite{Kudla1}, we have
 \[\omega_{\psi}(w)(\phi)(x)=\int_{M_2(F)}\psi(\langle  x,y\rangle)\phi(-y) ~d y=-\mathcal{F}(\phi)(-x).\]
From \cite[Lemma 2.2]{MS} and the description of the action of $U$ for the third $\SL_2$, it follows at once that 
 \[ 
 (\omega_{\psi})_{U,\psi} \cong C_c(\SL_2) 
 \] 
 as $(\GL_2 \times \GL_2)^{\det}$-modules.

  \begin{prop} \label{P:theta_3} 
 Let $\pi$ be an irreducible, Whittaker generic, representation of $\GL_2(F)$. Then 
 \[ 
\Theta(\pi):= (\omega_{\psi}\otimes \pi^{\vee})_{\SL_2} \cong \pi\otimes \pi
 \] 
 as $(G_2\times \GL_2)^{\det}/\SL_2\cong G_2$-modules. 
 \end{prop} 
 \begin{proof} Let $\SL_2$ be one of two in $\GL_2$. 
 Without loss of generality we can take the Schroedinger model of $\omega_{\psi}$ such that this $\SL_2$ plays the role of the symplectic group. 
 Since 
 \[ 
 \Theta(\pi)_{U,\psi} \cong ((\omega_{\psi})_{U,\psi} \otimes \pi^{\vee})_{\SL_2} \cong (C_c(\SL_2(F))\otimes \pi^{\vee})_{\SL_2} \cong \pi 
 \] 
 where the last isomorphism is Proposition \ref{P:theta_2}.  Now we can repeat the argument with the other $\SL_2$ in $G_2$. It follows that 
 $\pi\otimes \pi$ is a subquotient of $\Theta(\pi)$, and all other subquotients are non-generic on both factors, that is, a trivial representation when 
 restricted to $\SL_2 \times \SL_2$.  Working with three $\SL_2$, the trivial representation will split off of $\Theta(\pi)$ using the action of the Bernstein center, 
  unless $\pi$ is  Steinberg. In this case, it is a priori possible that $\Theta(\St)$ has $\St\otimes \St$ as a quotient and $1\otimes 1$ as the next subquotient. However, by 
 K\"uneth formula, 
 \[ 
 \mathrm{Ext}^1_{\SL_2 \times \SL_2}(1\otimes 1, \St\otimes \St) = 0 
 \] 
 and $1\otimes 1$ must be an $\SL_2^2$-invariant quotient of  $\Theta(\St)$.  Thus in all cases, if $\Theta(\pi)$ is larger than $\pi\otimes\pi$, then 
 $1\otimes 1 \otimes \sigma$ is  a quotient of $\omega_{\psi}$,  where $\sigma$ is an $\SL_2$-summand of $\pi$, a contradiction to the 
 fact that the small theta lift of the trivial representation of $\SL_2$ is the trivial representation of $\SL_2^2$ (\cite{Kudla1},\cite{BH_theta}).  
 \end{proof}

\subsection{Four $\SL_2$}  
 We have 
 \[ 
\SL_2\times \SL_2\times \SL_2\times \SL_2 /\mu_2^3 
 \] 
 in adjoint $D_4$. Let $\Pi$ be the minimal representation of $D_4$. Let $(U,\psi)$ be the 
 Whittaker datum for any one $\SL_2$. Then 
 \[ 
 \Pi_{U,\psi} =\omega_{\psi} 
 \]

  \begin{prop} \label{P:theta_4} 
 Let $\pi$ be an irreducible, Whittaker generic, representation of $\GL_2(F)$. Then 
 \[ 
\Theta(\pi):= (\Pi\otimes \pi^{\vee})_{\SL_2} \cong \pi\otimes \pi \otimes \pi 
 \] 
 as $(G_3\times \GL_2)^{\det}/\SL_2\cong G_3$-modules. 
 \end{prop} 
 \begin{proof} 
 Let $(U,\psi)$ be the Whittaker datum for any of the three $\SL_2$ in $G_3$.  Then 
 \[ 
 \Theta(\pi)_{U,\psi} \cong (\Pi_{U,\psi} \otimes \pi^{\vee})_{\SL_2} \cong (\omega_{\psi}\otimes \pi^{\vee})_{\SL_2} \cong \pi \otimes \pi 
 \] 
 where the last isomorphism is Proposition \ref{P:theta_3}.  It follows that 
 $\pi\otimes \pi \otimes \pi$ is a subquotient of $\Theta(\pi)$, and all other subquotients are non-generic on all factors, that is, a trivial representation when 
 restricted to $\SL_2 ^3$.  The trivial representation will split off of $\Theta(\pi)$ using the action of the Bernstein center, unless $\pi$ is 
 Steinberg. In this case, it is a priori possible that $\Theta(\St)$ has $\St \otimes \St \otimes \St$ as a quotient and $1\otimes 1\otimes 1 $ as the next subquotient. However, by 
 K\"uneth formula, 
 \[ 
 \mathrm{Ext}^1_{\SL_2 \times \SL_2\times \SL_2}(1\otimes 1 \otimes 1, \St\otimes \St\otimes \St) = 0 
 \] 
 and $1\otimes 1\otimes 1$ is a quotient of $\Theta(\St)$.  Thus in all cases, if $\Theta(\pi)$ is larger than $\pi\otimes\pi\otimes\pi$, then 
 $1\otimes 1 \otimes 1\otimes \sigma$ is  a quotient of $\omega_{\psi}$,  where $\sigma$ is an $\SL_2$-summand of $\pi$,
  a contradiction to Proposition \ref{P:mini_principal} which implies that 
 the small theta lift of the trivial representation of one $\SL_2$ is the trivial representation of $\SL_2^3$. 
 \end{proof}  
 
 Now we state a consequence of Proposition \ref{P:theta_4}  that we shall need later. In the adjoint $D_4$ we have a dual pair 
 \[ 
 \SL_2 \times (\SL_2 \times  \SL_2\times \SL_2 /\mu_2^2). 
 \] 
 If $\sigma$ is an irreducible representation of $\SL_2$ let 
 \[ 
  \mathrm{mini}(\sigma) =(\Pi\otimes{\sigma^{\vee}})_{\SL_2} 
  \] 
  be its theta lift to $(\SL_2  \times \SL_2\times \SL_2 /\mu_2^2)$. The correspondence described in Proposition \ref{P:theta_4} is a version of 
  this lift for groups of similitudes, as explained in   \cite{BGS_similitudes}.   
We are also using a simple observation that if the  big theta lift for similitudes correspondence is irreducible, then the big theta lift for the original correspondence is 
certainly semi-simple, so also irreducible, by Theorem 5.1 of \cite{BGS_similitudes};  see Lemma 5.3. of \cite{BGS_similitudes}.    Thus we have the following: 
  
  \begin{prop} \label{P:mini_theta} 
  Let $\pi$ be an irreducible, Whittaker generic, representation of $\GL_2(F)$. Let $\sigma$ be an irreducible $\SL_2$-summand of $\pi$. Then 
 $ \mathrm{mini}(\sigma)$ 
  is an $(\SL_2  \times \SL_2\times \SL_2 /\mu_2^2)$-irreducible summand of $\pi\otimes\pi\otimes\pi$. As $\sigma$ runs over all summands of $\pi$, 
  $\mathrm{mini}(\sigma)$ runs over all irreducible summands of $\pi\otimes\pi\otimes\pi$.  
 \end{prop}

  \section{Computing parameters of lifts of tempered representations} \label{S:parameters} 
 
 Here we determine $\Theta(\sigma)$ for supercuspidal $\sigma$ using a ``tower'' of theta correspondences, starting with the correspondence in 
 $D_4$, Proposition \ref{P:mini_theta} as the base case. To that end, let $\mathscr G_i$, $i=4, \ldots ,8$  denote one of the following 
 simply connected  Chevalley groups in the sequence 
 \[ 
 D_4 < D_5 < E_6 < E_7 < E_8 
 \] 
 where $i$ denotes the rank of the group.  
 Let $\mathscr P_i=\mathscr M_i \mathscr N_i $ be the  (standard) maximal parabolic in $\mathscr G_i$ such that $[\mathscr M_i, \mathscr M_i]=\mathscr G_{i-1}$.  
 We can arrange the dual pair $\SL_2 \times G_i$ in $\mathscr G_i$ such that 
$\SL_2$ is contained in the Levi $\mathscr M_i$. Let $P_i=M_iN_i$ be the centralizer of $\SL_2$ in $\mathscr P_i$. It is a maximal parabolic subgroup in $G_i$. 
Let $\bar {\mathscr N}_i$ and $\bar N_i$ denote the unipotent radical of respective opposite parabolic subgroups. Then $V_i={\bar{\mathscr N}_i}/ \bar N_i$ is abelian, and thus 
a module for $\SL_2 \times M_i$.  One checks that  $V_i=V_2 \otimes V_{M_i}$  for an $M_i$-module $V_{M_i}$.  
The cases are given by the following table, where the factor $\SL_2'$ acts trivially on $V_4$. The meaning of $r_i$ and $t_i$ will be explained below.
 \[ 
 \begin{array}{c||c|c|c|c} 
 \mathscr G_i & D_5 &  E_6 & E_7 & E_8  \\ \hline 
\mathscr M_i  & D_4 & D_5  & E_6  & E_7  \\ \hline 
M^{\mathrm {der}}_i &  \SL_2' \times  \SL_2''\times \SL_2'''& \SL_2'\times \SL_4  & \SL_6 &  \Spin_{12}  \\ \hline 
 V_{M_i} & 2 V'_2  &  V_4  & \wedge^2 V_6 &  V_{12}  \\  \hline
 t_i & 1 & 3/2 & 2 & 3\\ \hline
 r_i& -1 & -3/2 & -3 & -11/2\\
 \end{array}. 
 \] 
 
 Let $\varpi_{i}$ be the fundamental weight for $\mathcal G_i$ with respect to the maximal parabolic $\mathscr P_i$. Let 
 \[ 
 |\cdot |_{\mathscr M_i}: \mathscr M_i \rightarrow \mathbb R^+ 
 \] 
 be the positive character such that for any $\mathscr G_i$ co-root $\alpha^{\vee} : \mathbb G_m \rightarrow \mathscr M_i$ we have 
 \[ 
  |\alpha^{\vee}(t) |_{\mathscr M_i}= |t|^{\varpi_{i} (\alpha^{\vee})} 
 \] 
 
 Let $\Omega_i\subset V_i$ be the $\SL_2 \times M_i$-orbit consisting of highest weight vectors. 
 By Theorem 1.1 and Theorem 6.1 in  \cite{MS},  we have an exact sequence of $\SL_2 \times M_i$-modules 
 \[ 
 0\rightarrow C_c(\Omega_i) \rightarrow \Pi_{N_i} \rightarrow \Pi_{\mathscr N_i} \rightarrow 0, 
 \] 
 where 
\[ 
\Pi_{\mathscr N_i}=\Pi[\mathscr M_i] \otimes |\cdot|_{\mathscr M_i} ^{t_i} \oplus \mathbb C \otimes   |\cdot|_{\mathscr M_i} ^{s_i} 
\] 
 a direct sum of the minimal representation and the trivial representation of $\mathscr M_i$,  twisted by the powers of $|\cdot|_{\mathscr M_i}$ where 
 $t_i$ is  given in the  table  above. We note that our $s_i$ and $t_i$ are obtained form corresponding $s$ and $t$ in \cite{MS} by dividing by 
 lengths of fundamental characters $\varpi_i$.  
  The action of  $\SL_2 \times M_i$ on $C_c(\Omega_i)$ is geometric twisted by the character $|\cdot|_{\mathscr M_i} ^{s_i}$. 
  
  \begin{prop} \label{P:induction}  
  Let $\sigma$ be an irreducible, supercuspidal representation of $\SL_2$. Let $\Theta_i(\sigma)$ be its lift to $G_i$.  Then, as $M_i$-modules, 
  \[ 
  r_{P_i} (\Theta_i(\sigma)) \cong (\Pi[\mathscr M_i]\otimes \sigma^{\vee})_{\SL_2} \otimes |\cdot|_{M_i} ^{r_i} 
  \] 
  where $|\cdot|_{M_i}$ is defined analogously to $|\cdot|_{\mathscr{M}_i}$ but with respect to $G_i$ and $r_i$ is given in the table above.
   \end{prop} 
   \begin{proof} 
   The proposition follows from the filtration of $\Pi_{N_i}$, if we  can show that  $\sigma$ does not appear as a quotient of the bottom piece $C_c(\Omega_i)$.  
   But this is done exactly as in Proposition \ref{P:cuspidal_support}, (the discussion of the sub quotient $\Pi_1/\Pi_2$ there). As far as figuring out $r_i$ goes, we note 
   that the restriction of $|\cdot|_{\mathscr{M}_i}$ to $M_i$ is precisely $|\cdot|_{{M}_i}$ and then it remains to divide with the modular character $\delta_{P_i}^{1/2}$ to 
   determine $r_i$. 
 
   \end{proof}

\vskip 10pt 

We realize the root system $E_7$ siting in $\mathbb R^8$, in the standard fashion, where the simple roots are  
 \[ 
 \alpha_1=\frac12(e_1+e_8)-\frac12(e_2+e_3+e_4+e_5+e_6+e_7), ~\alpha_2=e_1+e_2, 
 \] 
 \[ 
  \alpha_3=e_2-e_1, ~\alpha_4=e_3-e_2, ~\alpha_5=e_4-e_3, ~\alpha_6=e_5-e_4, ~\alpha_7=e_6-e_5. 
 \] 
 
The following is the Dynkin diagram of $E_7$.

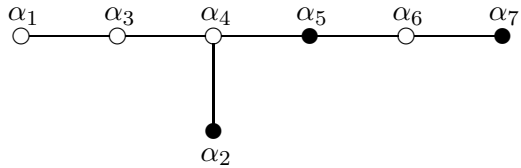
\begin{picture}(200,120)(-120,-15)

\put(79,73){\line(0,-1){30}}
\put(79,40){\circle*{6}}
\put(154,76){\line(1,0){30}}
\put(187,76){\circle*{6}}

\put(74,29){$\alpha_2$}

\put(02,82){$\alpha_1$}

\put(38,82){$\alpha_3$}

\put(74,82){$\alpha_4$}

\put(110,82){$\alpha_5$}

\put(146,82){$\alpha_6$}

\put(182,82){$\alpha_7$}

\put(07,76){\circle{6}}
\put(10,76){\line(1,0){30}}

\put(43,76){\circle{6}}
\put(46,76){\line(1,0){30}}
\put(79,76){\circle{6}}

\put(82,76){\line(1,0){30}}
\put(115,76){\circle*{6}}

\put(118,76){\line(1,0){30}}
\put(151,76){\circle{6}}

\end{picture}

We consider  $G_i$, for $i=6,7,8$ to be in the sequence $A_5 < D_6 <E_7$ where $A_5$ is picked to contain the three black dots.  
Let $Q_i=L_iU_i$ be the standard parabolic in $G_i$ corresponding to the three simple roots, that is, $[L_i,L_i]=\SL_2^3$. 
We shall prove that $\Theta_i(\sigma)$, where $\sigma$ is a tempered representation of $\SL_2$, are Langlands 
 quotients of standard modules supported on $Q_i$.  To that end, let 
 $T_i$ be the connected component of the center of $L_i$. It is a split torus of the dimension $2,3$ and $4$, respectively. One checks 
that $[L_i,L_i]\cap T_i= \mu_2^2$.  It follows that 
\[ 
L_i/T_i\cong \SL_2^3/\mu_2^2 
\] 
and $\mini(\sigma)$ is thus a representation of $L_i$. We shall prove that $\Theta_i(\sigma)$ is the 
Langlands quotient of $\Ind_{Q_i}^{G_i}(\mini(\sigma) \otimes \chi_i)$ where $\chi_i$ is a positive character of $L_i$. We shall 
describe $\chi_i$ by giving the weight $\lambda_i$ such that 
\[ 
\chi_i(\alpha^{\vee}(t))= |t|^{\lambda_i(\alpha^{\vee})} 
\] 
for any co-root $\alpha^{\vee}$.  At the same time we shall verify that the Langlands parameter of $\Theta_i(\sigma)$ is 
\[ 
\varphi_{G_i} :  WD_F  \rightarrow  G_i^{\vee}, 
\] 
described in the introduction in terms of the principal $\SL_2$ in $\Aut(J_i)$. However, this $\SL_2$ is also principal in $S_i^{\vee}$, the dual group of the 
Levi factor $S_i$ of the Siegel parabolic from Section \ref{S:principal} and denoted $M$ there. 
Hence $\lambda_i$ should be a $W_{G_i}$-conjugate of $\rho_{S_i}$. 
We give details on case by case basis, using reverse engineering.

\vskip 10pt 

\noindent 
Case $G_6=A_5$. Here $S_6 =A_2 \times A_2$ and $\rho_{S_6}=(1,0,1,-1,0,1)$ which, after using the Weyl group $W_{G_6}$,  
can be put in the positive position 
\[ 
\lambda_6=(1,-1,0,0,1,1).  
\] 
Observe that $\lambda_6$ is perpendicular to the three black simple roots. Thus it defines a positive character $\chi_6$ of $L_6$ such that 
$\Ind_{Q_6}^{G_6}(\mini(\sigma) \otimes \chi_6)$ is a standard module. The parameter of its 
  Langlands quotient is  $\varphi_{G_6}$.  

\vskip 10pt 

\noindent 
Case $G_7=D_6$. Here $S_7=A_5$, but this $A_5$ is not the one containing the three black roots. 
  Hence $\rho_{S_7}=(-5,-3,-1,1,3,5)/2$ which, after using the Weyl group $W_{G_7}$, 
can be put in the positive position 
\[ 
\lambda_7=(-1,1,3,3, 5,5)/2.  
\] 
Then $\lambda_7$ is perpendicular to the three black simple roots. Thus it defines a positive character $\chi_7$ of $L_7$ such that 
$\Ind_{Q_7}^{G_7}(\mini(\sigma) \otimes \chi_7)$ is a standard module. The parameter of its 
 Langlands quotient is  $\varphi_{G_7}$.  

.  

\vskip 10pt 

\noindent 
Case $G_8=E_7$. Here $S_8=D_6$ and  $\rho_{S_8}=(0,1,2,3,4,-4,-4,4)$ which, after using the Weyl group $W_{G_8}$ 
can be put in the positive position 
\[ 
\lambda_8=(-1,1,3,3,5,5,-11,11)/2.
\] 
Again, $\lambda_8$ is perpendicular to the three black  simple roots. Thus it defines a positive character $\chi_8$ of $L_8$ such that 
$\Ind_{Q_8}^{G_8}(\mini(\sigma) \otimes \chi_8)$ is a standard module. The parameter of its 
  Langlands quotient is  $\varphi_{G_8}$.

\begin{prop} Assume that $\sigma$ is tempered. 
In each of the three cases above,  $\Theta_i(\sigma)$ is the Langlands quotient of $\Ind_{Q_i}^{G_i}(\mini(\sigma) \otimes \chi_i)$. 
\end{prop} 
\begin{proof} 
To keep the exposition simple we shall assume that $\sigma$ is supercuspidal, otherwise we already determined $\Theta_i(\sigma)$.  
In this case we shall prove that  $r_{Q_i}(\Theta_i(\sigma)) \cong \mini(\sigma) \otimes \chi^{-1}_i$.  
This implies that $\Theta_i(\sigma)$ is the unique sub of $\Ind_{Q_i}^{G_i}(\mini(\sigma) \otimes \chi^{-1}_i)$, an equivalent statement.  

Now observe that the normalized Jacquet functor $r_{Q_i}$ is simply the composite 
\[ 
r_{Q_i}= r_{P_4}\circ \ldots \circ r_{P_i}
\] 
hence $r_{Q_i}(\Theta_i(\sigma)) \cong \mini(\sigma) \otimes \chi^{-1}_i$ will be established by induction on $i$. 

We leave as an exercise to the reader to check this for $G_6=A_5$. We move on to $G_7=D_6$.  Proposition \ref{P:induction} 
says that $r_{P_7}(\Theta_7(\sigma))$ is $\Theta(\sigma_6)$  twisted by a positive character of the Levi $M_7=A_5$ corresponding to 
$-3$ times the fundamental character for the pair $(D_6,A_5)$.  Since the difference 
\[ 
\lambda_6-\lambda_7 = (1,-1,0,0,1,1)- (-1,1,3,3, 5,5)/2 = -(-3,3,3,3,3,3)/2  
\] 
is precisely that character, the proposition follows for $G_7=D_6$. We move on to $G_8=E_7$.  Proposition \ref{P:induction} 
says that $r_{P_8}(\Theta_8(\sigma))$ is $\Theta(\sigma_7)$   twisted by a positive character of the Levi $M_8=D_6$ corresponding to 
$-11/2$ times the fundamental character for the pair $(E_7,D_6)$.  Since the difference  
\[ 
\lambda_7-\lambda_8 = (-1,1,3,3, 5,5,0,0)/2  - (-1,1,3,3,5,5,-11,11)/2 =-(0,0,0,0,0,0, -11,11)/2  
\] 
is precisely that character, the proposition follows for $G_8=E_7$. 
\end{proof}

\noindent\textbf{Acknowledgment} A part of this work was done during the conference on the occasion of 70-th birthday of Marko Tadi\'c at 
the University of Zagreb. The authors would like to thank Neven Grbac for hospitality and wonderful organization.  We 
thank Wee Teck Gan and Chen Wan for explaining the connection of this work to the 
program of Ben-Zvi, Sakellaridis and Venkatesh.

\bibliography{E8_E7}

\providecommand{\bysame}{\leavevmode\hbox to3em{\hrulefill}\thinspace}
\providecommand{\MR}{\relax\ifhmode\unskip\space\fi MR }
\providecommand{\MRhref}[2]{%
  \href{http://www.ams.org/mathscinet-getitem?mr=#1}{#2}
}
\providecommand{\href}[2]{#2}
\begin{thebibliography}{MWZ24}

\bibitem[AG17]{Atobe_Gan_theta_Temp}
Hiraku Atobe and Wee~Teck Gan, \emph{Local theta correspondence of tempered
  representations and {Langlands} parameters}, Invent. Math. \textbf{210}
  (2017), no.~2, 341--415 (English).

\bibitem[BGS24]{BGS_similitudes}
Petar Baki{\'c}, Wee~Teck Gan, and Gordan Savin, \emph{Similitude exceptional
  theta correspondences}, Rad Hrvatske Akademije Znanosti i Umjetnosti
  \textbf{558} (2024), no.~28, 193--222.

\bibitem[BH21]{BH_theta}
Petar Baki{\'c} and Marcela Hanzer, \emph{Theta correspondence for
  {{\(p\)}}-adic dual pairs of type {I}}, J. Reine Angew. Math. \textbf{776}
  (2021), 63--117 (English).

\bibitem[BS12]{Berndt_Schmidt_Jacobi}
Rolf Berndt and Ralf Schmidt, \emph{Elements of the representation theory of
  the {Jacobi} group}, reprint of the 1998 original ed., Mod. Birkh{\"a}user
  Class., Basel: Birkh{\"a}user, 2012 (English).

\bibitem[BZ76]{BZ_reps_GL}
I.~N. Bernshte{\u{\i}}n and A.~V. Zelevinski{\u{\i}}, \emph{Representations of
  the group {{\(\mathrm{GL}(n,F)\)}} where {{\(F\)}} is a non-{Archimedean}
  local field}, Russ. Math. Surv. \textbf{31} (1976), no.~3, 1--68 (English).

\bibitem[BZSV]{BZSV}
D.~Ben-Zvi, Y.~Sakellaridis, and A.~Venkatesh, \emph{Relative {L}anglands
  duality}.

\bibitem[GG06]{GG}
W.~T. Gan and N.~Gurevich, \emph{Nontempered {A}-packets of {$G_2$}}, American
  Journal of Mathematics (2006), no.~170, 1105--1185.

\bibitem[GS99]{GS_E7}
Wee~Teck Gan and Gordan Savin, \emph{The dual pair ${G}_2\times {PU}_3({D})$,
  p-adic case}, Canadian Journal of Mathematics \textbf{51} (1999), no.~1,
  130--146.

\bibitem[GS05]{GS_minimal}
\bysame, \emph{On minimal representations definitions and properties},
  Represent. Theory \textbf{9} (2005), 46--93 (English).

\bibitem[HS20]{HS_Jordan_Eisenstein}
Marcela Hanzer and Gordan Savin, \emph{Eisenstein series arising from {Jordan}
  algebras}, Can. J. Math. \textbf{72} (2020), no.~1, 183--201 (English).

\bibitem[Kno06]{Knop_symplectic}
Friedrich Knop, \emph{Classification of multiplicity free symplectic
  representations}, J. Algebra \textbf{301} (2006), no.~2, 531--553 (English).

\bibitem[KS90]{KS_minimal}
D.~Kazhdan and G.~Savin, \emph{The smallest representation of simply laced
  groups}, Festschrift in honor of {I}. {I}. {Piatetski}-{Shapiro} on the
  occasion of his sixtieth birthday, {Pt}. {I}: {Papers} in representation
  theory, {Pap}. {Workshop} {{\(L\)}}-{Functions}, {Number} {Theory},
  {Harmonic} {Anal}., {Tel}-{Aviv}/{Isr}. 1989, {Isr}. {Math}. {Conf}. {Proc}.
  2, 209-223 (1990)., 1990.

\bibitem[KS15]{KS_jordan}
T.~Kobayashi and G.~Savin, \emph{Global uniqueness of small representations},
  Mathematische Zeitschrift \textbf{281} (2015), no.~1-2, 215--239.

\bibitem[Kud96]{Kudla1}
Stephen~S. Kudla, \emph{Notes on the local theta correspondence}, unpublished
  notes (1996).

\bibitem[MS97]{MS}
K.~Magaard and G.~Savin, \emph{Exceptional {{\(\Theta\)}}-correspondence. {I}},
  Compos. Math. \textbf{107} (1997), no.~1, 89--123 (English).

\bibitem[Mui08]{Muic}
Goran Mui{\'c}, \emph{On the structure of theta lifts of discrete series for
  dual pairs {{\((\text{Sp}(n), \text{O}(V))\)}}}, Isr. J. Math. \textbf{164}
  (2008), 87--124 (English).

\bibitem[MWZ24]{MWZ}
Zhengyu Mao, Chen Wan, and Lei Zhang, \emph{{BZSV} duality for some strongly
  tempered spherical varieties}, arXiv:2310.17837v3 (2024), 30 pages.

\bibitem[Rum97]{Rumelhart}
Karl~E. Rumelhart, \emph{Minimal representations of exceptional {{\(p\)}}-adic
  groups}, Represent. Theory \textbf{1} (1997), 133--181 (English).

\bibitem[Wei03]{Weissman_FJ_small}
Martin~H. Weissman, \emph{The {Fourier}-{Jacobi} map and small
  representations.}, Represent. Theory \textbf{7} (2003), 275--299 (English).

\end{thebibliography}
\bibliographystyle{amsalpha}

\end{document}